\documentclass[12pt]{amsart}

\usepackage{amsmath,amsfonts,amssymb}
\usepackage{amsthm,amsfonts,amssymb}
\usepackage{graphics}
\usepackage{pifont}
\usepackage[usenames,dvips]{color}
\usepackage{amsrefs}
\def\({\left(}
\def\){\right)}
\def\be {\begin{equation}}
\def\en{\end{equation}}

\let\nc\newcommand

\nc{\Ref}[1]{{\rm(\ref{#1})}}

\nc{\BB}{{A_{p,q}}} \nc{\bean}{\begin{eqnarray}} \nc{\eean}{\end{eqnarray}}
\nc{\bea}{\begin{eqnarray*}} \nc{\eea}{\end{eqnarray*}}

\def\cd{{\mathcal D}}

\def\cg{{\mathcal G}}

\def\cs{{\mathcal A}}

\theoremstyle{plain}
\begingroup 
\newtheorem{thm}{Theorem}[section]   
\newtheorem*{thm*}{Theorem}          
\newtheorem*{cor*}{Corollary}        
\newtheorem{cor}[thm]{Corollary}     

\newtheorem{lem}[thm]{Lemma}         
\newtheorem{prop}[thm]{Proposition}  
\endgroup

\theoremstyle{definition}

\newtheorem*{rem*}{Remark}
\newtheorem*{ack*}{Acknowledgment}

\theoremstyle{remark}
\newtheorem{rem}[thm]{Remark}        %

\theoremstyle{definition}

\numberwithin{equation}{section}

\begin{document}

\title[Euler adic]
{The Euler adic dynamical system and path counts in the Euler graph}

\author{K. Petersen}
\address{Department of Mathematics, University of North Carolina
at Chapel Hill\\ Chapel Hill, NC 27599-3250, USA} \email{petersen@math.unc.edu}

\author{A. Varchenko}
\address{Department of Mathematics, University of North Carolina
at Chapel Hill\\ Chapel Hill, NC 27599-3250, USA} \email{anv@math.unc.edu}
\thanks{The research of
the second author was supported in part by NSF grant DMS-0555327}

\date{\today}

\keywords{adic transformation, invariant measure, ergodicity, Eulerian numbers} \subjclass{Primary:
37A05, 37A25, 05A10, 05A15; Secondary: 37A50, 37A55}

\begin{abstract}
We give a formula for generalized Eulerian numbers, prove monotonicity of sequences of certain ratios
of the Eulerian numbers,
and apply these results to obtain a new proof that the natural symmetric measure for the
Bratteli-Vershik {dynamical} system based on the Euler graph is the unique fully supported invariant
ergodic Borel probability measure. {Key ingredients of the proof are
  a two-dimensional induction argument and a one-to-one correspondence
  between most paths from two vertices at the same level to another
  vertex.}
\end{abstract}

\maketitle

\section{Introduction}
\label{sec_intro} The Euler graph is an infinite,
 directed graph with vertices $(i,j)$, $i,j\geq 0$,
with $j+1$ edges from $(i,j)$ to $(i+1,j)$ and $i+1$ edges from $(i,j)$ to $(i,j+1)$; see Figure
\ref{fig:eulergraph}. {The number $i+j$ is called the level of
  $(i,j)$.}

A {\em generalized Eulerian number} $A_{p,q}(i,j)$ is the number of
paths in the graph from $(p,q)$ to $(p+i,q+j)$. We prove that \bean
\label{f1} \BB(i,j)\ = \ \sum_{t=0}^i\,(-1)^{i-t}{p+q+t+1 \choose t}
{p+q+i+j+2 \choose i-t} (p+1+t)^{i+j}\ \eean and \bean \label{f2}
\frac{\BB(i,j+1)} {A_{p,q-1}(i,j+1)} \ \leq \
\frac{\BB(i,j)}{A_{p,q-1}(i,j)}\ \leq \frac{q+j}{q+1+j} \
\frac{\BB(i+1,j)}{A_{p,q-1}(i+1,j)}\ . \eean {As a corollary we show
that \bean \label{formula 1} \frac{\BB (i,j)}{A_{0,0}(p+i,q+j)}
\quad {\rm tends\ to} \quad {\frac1{(p+q+1)!}} \eean
 as both $i,j$ tend to infinity. }

{These results are motivated by continuing study of the adic dynamical
  system associated with the Euler graph. They yield a new proof of
  the fact that the natural symmetric measure for the Bratteli-Vershik
  dynamical system based on the Euler graph is the unique fully
  supported invariant ergodic Borel probability measure.}

The Euler adic system is a particularly interesting nonstationary
Bratteli-Vershik (or adic) system based on an infinite directed
graded graph with remarkable combinatorial properties. How such
systems arise from reinforced walks on graphs is explained in
\cite{FP2}. For the viewpoint of urn models, see for example
\cite{Flajolet2006}*{p. 68 ff.}. A first step in studying the Euler
adic system, or the associated $C^*$ algebra, is the identification
of the adic-invariant measures (sometimes called central measures or
traces), namely those that give equal measure to each cylinder set
determined by an initial path segment from the root vertex to
another fixed vertex. In \cite{BKPS} it was proved by a
supermartingale argument that the natural {\em
  symmetric measure}, which assigns equal measure to all cylinders of
 the same length, is ergodic. In \cite{FP} this result was strengthened
 by using a coding of paths by
 permutations to show that in fact the symmetric measure is the unique fully supported ergodic
 probability measure for this system.
{
 We found out recently that the paper
 \cite{Gnedin-O} contains related results, arrived at by different arguments and
 including also identification
 of all the other (partially supported) ergodic measures, and that a version of Formula
 (\ref{f1})
 appears in \cite{CarlitzScoville1974}. Although the connection with
 path counting is not made explicitly in \cite{CarlitzScoville1974},
 it seems that it is not too far from Formula (6.2) of that paper to
 our Formula (\ref{f1}).}
 When one seeks to study higher-dimensional versions of the Euler adic system, the
 coding by permutations is no longer available; therefore we have developed
 a proof via a different approach, which we present here.
 This proof also yields a stronger result, namely identification of the
 generic points for the symmetric measure and indeed a
 ``directional unique ergodicity" property such as
 was established in \cites{M,MP} for the Pascal adic system.

For background on adic systems, we refer to \cites{Vershik3, VL,KV, HPS,GPS,PS,M,MP,BKPS}. A Bratteli
diagram is an infinite directed graded graph. At level $0$ there is a single vertex, $R$, called the
{\em root}.  {At each level $n \geq 1$ there are finitely
  many vertices.}  There are edges only from vertices at level $n$ to
{vertices at} level $n+1$, for all $n$. Each vertex has only
finitely many edges leaving or entering it. { For nontriviality we
assume that
  each vertex has at least one edge leaving it and, except for $R$, at
  least one edge entering it.  The set of edges entering each vertex
  is ordered. Often when the diagram is drawn we assume that the edges
  are ordered from left to right.} The phase space $X$ of the
dynamical system based on the diagram is the set of infinite paths
that begin at $R$. {The space} $X$ is a compact metric space with the
distance between two paths that agree on exactly the first $n$ levels
being $1/2^n$. A {\em cylinder set} is { the set of all paths with a
  specified initial} segment of finite length.  {That length is called
  the {\em length} of the cylinder.}  Cylinder sets are open and closed and
form a base for the topology of $X$. We define a partial order {on the
  set $X$.}  Two paths $x$ and $y$ are {\em comparable} if they
coincide after some {level. In this case} we determine which of $x$
and $y$ is larger by comparing the last edges that differ. In this
partial order there is a set $X_{\max}$ of maximal paths {and a set}
$X_{\min}$ of minimal paths. The {\em adic transformation} $T: X
\setminus X_{\max} \to X \setminus X_{\min}$ is defined by letting
$Tx$ be the {smallest $y$ such that $y > x$}. Both $T$ and $T^{-1}$
are continuous where defined.

We recall quickly now how counting paths between vertices in a
Bratteli diagram is related to the identification of ergodic
invariant measures. For vertices $P, Q$ of a directed graph, denote
by $\dim(P,Q)$ {the number of paths from $P$ to $Q$.} Cylinder
{sets} determined by initial paths terminating at a common {vertex}
are mapped to one another by powers of $T$ and so {they} must be
assigned equal measure by any invariant measure. {For a path $x \in
X$, denote
  by $x_n$ the vertex of $x$ at level $n$.
  In the case of the Pascal and Euler graphs, we give $x_n$ the
 rectangular coordinates $(i_n,j_n)$.}
 It can be proved (see
\cites{Vershik1,VK}) by using either the Ergodic Theorem or Reverse
Martingale Theorem that if $\mu$ is a $T$-invariant ergodic Borel
probability measure on $X$ and $C$ is any cylinder set terminating
at a vertex $P$, then
\be \label{dimlim}
 {
 \mu(C)=\lim_{n \to \infty}
\frac{\dim(P,{x_n} )}{\dim(R,{x_n})} \quad\text{for $\mu$-almost
every}\ x \in X. }
 \en

{In this paper we show that if $\mu$ is a fully supported ergodic
  measure} for
{the adic system on the Euler graph}, then for $\mu$-almost every $x$ the limit in
(\ref{dimlim}) {has the same value for any two cylinders of
  the same length.  Consequently we show that } there is only one
fully supported ergodic $T$-invariant measure for the Euler adic
system, namely, the symmetric measure, which assigns equal measure to
all cylinder sets of a given length. {In particular, if $n$ is the
  length of a cylinder $C$, then $\mu(C)=1/(n+1)!$\,.}

\begin{thm}
\label{mainth} In the Euler graph, for each vertex $P$ the limit
\bean \label{Dimlim} \lim_{n \to
\infty} \frac{\dim(P,(i_n,j_n))}{\dim(R,(i_n,j_n))}
\eean
 exists for all infinite paths $(i_n,j_n),n \geq 0,$
 for which $i_n$ and $j_n$ are
 unbounded. Moreover, this limit is
constant as $P$ varies over the vertices at any fixed level.
\end{thm}

 {Theorem \ref{mainth}, the new proof of which is
the main point of this paper, is a corollary of Formula \Ref{formula
1}.}

{In Section \ref{sec-graph} we introduce the recurrence relations for
 generalized Eulerian numbers. In
Sections \ref{sec rec relation}  and \ref{sec_mon} we prove Formulas \Ref{f1} and \Ref{f2},
respectively. In Section \ref{Sec limit theorems} we prove that
\bean \label{f4} \frac{\BB
(i,j)}{A_{p,q-1}(i,j)} \to \infty \qquad \text{ and} \qquad \frac{\BB (i,j)}{A_{p-1,q}(i,j)} \to \infty
\eean
 as both $i,j$ tend to infinity.}

 In Section \ref{sec good}, we consider
the set $\cs_{p,q}(i,j)$ of all paths from $(p,q)$ to $(p+i,q+j)$ and introduce a subset
 $\cg_{p,q}(i,j)\subset \cs_{p,q}(i,j)$ of ``good" paths.
 {
 Good paths are defined to be those which use each of a particular set
 of labels at least once (see Section \ref{sec good}). These
 are designed to substitute, in a way that will extend to
 higher-dimensional Euler adic systems, for the path-coding
 permutations in the two-dimensional case with
 only singleton ``clusters" in \cite{FP} which are predominant.}
 We show that almost all paths are good
asymptotically as both $i,j$ tend to infinity.

{ In Section \ref{sec identif} we show that the number of good paths from $(p,q)$ to
$(i,j)$ is equal to the number of good paths from
$(p',q')$ to $(i,j)$, if $p+q=p'+q'$ and $i,j \geq
p+q+1$. This will complete the proof of Formula \Ref{formula 1} and Theorem \ref{mainth}.}

\begin{ack*}
 {
{ The authors thank} Sarah Bailey Frick and Xavier M\' ela for
conversations on this topic and Thomas Prellberg for finding
reference \cite{CarlitzScoville1974}.}
\end{ack*}

\section{The Euler graph}
\label{sec-graph} The Euler graph is an infinite, directed graph. The vertices of the graph are labeled
by pairs of nonnegative integers $(i,j)$. The list of edges is given by the rule: for any $(i,j)$ there
are $j+1$ edges from $(i,j)$ to $(i+1,j)$ and $i+1$ edges  from $(i,j)$ to $(i,j+1)$; see Figure
\ref{fig:eulergraph}.

\begin{figure}[h]
  \scalebox{.60}{\includegraphics{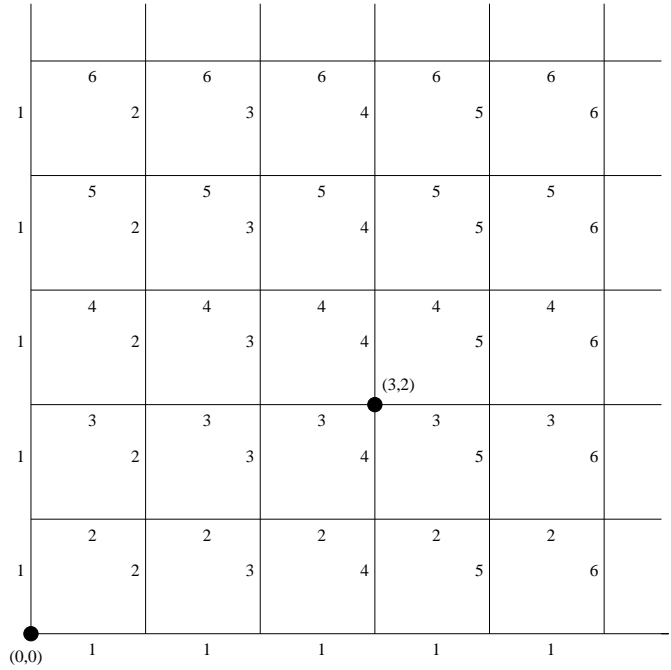}}
  \caption{The first part of the Euler graph.
  Numbers along edges indicate multiple edges.}
  \label{fig:eulergraph}
  \end{figure}

 Define a {\em generalized Eulerian number} $A_{p,q}(i,j)$ to be
the number of paths in the graph from $(p,q)$ to $(p+i,q+j)$.  For fixed $p,q$, the
 numbers satisfy the recurrence relation
\be
\label{erel}
\begin{gathered}
A_{p,q}(i,j)\ =\
(j+q+1)A_{p,q}(i-1,j)\ +\ (i+p+1)A_{p,q}(i,j-1)
\end{gathered}
\en and initial conditions \bean \label{initial} A_{p,q}(i,0) \ = \ (q+1)^{i}\ , \qquad A_{p,q}(0,j)\ =
\ (p+1)^{j}\ , \qquad i,j > 0\ . \phantom{aa} \eean A closed form  is known  \cite{Comtet} for
$A_{0,0}(i,j)$:
 \bea
 A_{0,0}(i,j)\ =\ \sum_{t=0}^i(-1)^{i-t}\, \binom{i+j+2}{i-t} (1+t)^{i+j+1}\ .
\eea
We develop a closed form for all  $A_{p,q}(i,j)$ in the next section.

\section{{The} recurrence relation}
\label{sec rec relation}

We say that a collection of numbers { $\{A(i,j), \ i,j\geq 0,\, (i,j)\neq (0,0)\}$}
 satisfies recurrence relation
\Ref{erel}, if the relation
 holds for any $i,j>0$. The numbers
\bea A(i,0), \ {} A(0,j)\quad {\rm with} \quad i,j > 0 \eea
 will be called the {\em initial conditions}
of the collection.

It is clear that for any collection of numbers {$\{a(i,0),\ a(0,j), \,i,j\geq 0,\, (i,j)\neq (0,0)\}$,}
there exists a unique collection {$\{A(i,j), \ i,j\geq 0,\, (i,j)\neq (0,0)\}$} satisfying the
recurrence relation \Ref{erel} and initial conditions
\bea A(i,0)=a(i,0)\ , \qquad A(0,j)=a(0,j)\ .
\eea

\begin{thm}
\label{Bount nw} The collection {$\{\BB(i,j),\ i,j\geq 0,\, (i,j)\neq (0,0)\} $ } satisfying the
recurrence relation \Ref{erel} and initial conditions \Ref{initial} is
given by the following formula:

\be \label{BB nw} \BB(i,j)\ = \ \sum_{t=0}^i\,(-1)^{i-t}{p+q+t+1 \choose t} {p+q+i+j+2 \choose i-t}
(p+1+t)^{i+j}\ .
\en
\end{thm}

\begin{proof}

Let $t=0,1,\dots$. For $i,j \geq 0$, define the numbers \bea a_t(i,j)\ &=&\ (-1)^{i-t} {p+q+i+j+2
\choose i-t} (p+1+t)^{i+j}
\\
&=&\ (-1)^{i-t}
\frac{\Gamma(p+q+i+j+3)} {\Gamma(i-t+1)\Gamma(p+q+j+t+3)}
(p+1+t)^{i+j}
\ , \eea
where $\Gamma(x)$ is
Euler's gamma function.
The last expression,  in particular, says that
\bea \label{a zero}
a_t(i,j)\ =\
0 \qquad {\rm if} \quad j\geq 0,\quad t> i\geq 0\ .
\eea

\begin{lem}
\label{lem on b} For any $t\geq 0$, the collection
{$\{a_t(i,j),\ i,j\geq
0,\,(i,j)\neq(0,0)\}$}
 satisfies recurrence relation
\Ref{erel} with the initial conditions
\bea
 a_t(0,j)\ &=&\ 0\  \quad {\rm if}\ t>0 \ ,
\quad\quad\phantom{a} a_t(0,j)\ =\ (p+1)^j\  \quad {\rm if}\ t=0\ ,
\\
a_t(i,0)\ &=&\ (-1)^{i-t} {p+q+i+2 \choose i-t} (p+1+t)^{i}
\\
& =&
\
(-1)^{i-t} \frac{\Gamma(p+q+i+3)}
{\Gamma(i-t+1)\Gamma(p+q+t+3)} (p+1+t)^{i}\ ,
\eea
for all $i,j>0$.
\end{lem}
The lemma is proved by direct verification.

\medskip

By Lemma \ref{lem on b}  any linear combination\ $\{\sum_{t= 0}^\infty\, C_t \,a_t(i,j),\ {}
 i,j\geq 0\}$
is well defined and satisfies relation \Ref{erel}. Therefore,
let us look for $\BB(i,j)$ in the form
\bea
\BB(i,j)\ = \ \sum_{t=0}^\infty\,C_t\,(-1)^{i-t}
 {p+q+j+i+2 \choose i-t} (p+1+t)^{i+j} \ .
\eea The constants $C_t$ can be found from the equations
\bea \sum_{t=0}^i\,C_t\,(-1)^{i-t} {p+q+i+2
\choose i-t} (p+1+t)^{i} \ =\ (q+1)^i\ . \eea

\begin{lem}
\label{lem on costs} For all $t\geq 0$, \bea C_t\ = \ {p+q+t+1\choose t}\ . \eea
\end{lem}
\begin{proof}
We need to show that for any $i>0$, \bean \label{main} \sum_{t=0}^i\,\,(-1)^{i-t} {p+q+t+1\choose t}
{p+q+i +2\choose i-t} (p+1+t)^{i} \ =\ (q+1)^i\ .
\eean
The right and left sides of \Ref{main} are
polynomials in
$q$ of degree $i$. To see that the two polynomials are equal it is enough to check that
they are equal at \bea q \ =\  -(p+2),\
-(p+3),\ \dots\ ,\ -(p+2+i)\ . \eea
If $q=-(p+2+t)$ for $0\leq
t\leq i$,
\ then the left side has exactly one nonzero summand. Moreover, that nonzero summand equals
the right-hand side polynomial at $q=-(p+2+t)$.
More precisely, denote by $L(q)$ and $R(q)$ the left
and right sides of \Ref{main}. Then
\bea
L(q)\ = \ (\prod_{j=0}^{i} \,(p+q+j+2)\,) \,
\sum_{t=0}^i\,\,(-1)^{i-t} \,\frac 1 {p+q+t+2}\,
\frac 1{t!\,(i-t)!}\, (p+1+t)^{i} \ . \eea
Hence, for
$0\leq t\leq i$, we have
\bea L(-p-2-t)\ = \ (-1)^t \,t!\,(i-t)!\, \,(-1)^{i-t} \, \frac
1{t!\,(i-t)!}\, (p+1+t)^{i} \ = \ (-1)^i\,(p+1+t)^{i} \ .
\eea
 But $R(-p-2-t) =   (-p-1-t)^i.$
\end{proof}
Theorem \ref{Bount nw} follows from Lemma \ref{lem on costs}.
\end{proof}

\begin{cor}
\label{cor interchange i,j}
The collection $\{\BB(i,j),\ i,j\geq 0\}$ satisfying the recurrence
relation \Ref{erel} and initial conditions \Ref{initial} is given also by the following formula:
\be
\label{B interchange i,j} \BB(i,j)\ = \ \sum_{t=0}^j\,(-1)^{j-t}{p+q+t+1 \choose t}
{p+q+i+j+2 \choose
j-t} (q+1+t)^{i+j} .
\en
\end{cor}
\begin{proof}
The corollary follows from Theorem \ref{Bount nw}
 due to the symmetry of \Ref{erel} {and} \Ref{initial}
with respect to the transformation $(p,q,i,j) \to (q,p,j,i)$.
\end{proof}

\section{Monotonicity of ratios}
\label{sec_mon}

\begin{thm}
\label{thm monotonicity} For $p,q\geq 0$, assume that a collection of positive numbers
{$\{a(i,j),\ i,j\geq 0,\,(i,j)\neq(0,0)\}$}
 satisfies
\bea
a(i,j)\ =\ (j+q+1)a(i-1,j)\ +\ (i+p+1)a(i,j-1) \eea and a collection of positive numbers
{$\{b(i,j),\ i,j\geq 0,\,(i,j)\neq(0,0)\}$}
 satisfies
\bea
b(i,j)\ =\ (j+q)b(i-1,j)\ +\ (i+p+1)b(i,j-1) . \eea Assume that the initial conditions of these
collections satisfy the inequalities: \bean \label{initial ineq} \frac{a(0,j+1)}{b(0,j+1)} & \leq &
\frac{a(0,j)}{b(0,j)}\ , \qquad
 \phantom{aaaaaaaaa}
j>0\ ,
\\
\frac{a(0,1)}{b(0,1)} & \leq & \frac{q}{q+1} \ \frac{a(1,0)}{b(1,0)} \ , \notag
\\
\frac{a(i,0)}{b(i,0)} & \leq & \frac{q}{q+1} \ \frac{a(i+1,0)}{b(i+1,0)}\ , \qquad \ {} i>0\ . \notag
\eean Then for any $i,j\geq 0$ we have \bean \label{monot ineq} \frac{a(i,j+1)}{b(i,j+1)} \ \leq \
\frac{a(i,j)}{b(i,j)}\ \leq \frac{q+j}{q+1+j} \ \frac{a(i+1,j)}{b(i+1,j)}\ . \eean
\end{thm}

\begin{cor}
\label{cor monot} The generalized Eulerian numbers satisfy the inequalities
\bean \label{euler ineq}
\frac{\BB(i,j+1)}{A_{p,q-1}(i,j+1)} \ \leq \ \frac{\BB(i,j)}{A_{p,q-1}(i,j)}\ \leq \frac{q+j}{q+1+j} \
\frac{\BB(i+1,j)}{A_{p,q-1}(i+1,j)}\
\eean
for all $p,q-1, i,j\geq 0$.
\end{cor}

\medskip
\noindent {\it Proof of Theorem \ref{thm monotonicity}.} The proof is
by induction. Denoting \bea x=(i,j+1) , &\qquad & w=(i+1,j+1) ,
\\ y=(i,j) , \phantom{aaa} &\qquad & z=(i+1,j) , \eea we assume that
\bean \label{1} \frac {a(x)}{b(x)}\ \leq\ \frac {a(y)}{b(y)} \ \leq \ \frac{q+j}{q+1+j}\ \frac
{a(z)}{b(z)}\ \eean and prove \bean \label{2} \frac {a(w)}{b(w)}\ \leq\ \frac {a(z)}{b(z)} \quad {\rm
and} \quad \frac
      {a(x)}{b(x)} \ \leq\ \frac{q+1+j}{q+2+j}\ \frac
      {a(w)}{b(w)}\ . \eean In fact, we will use not \Ref{1} but its
      corollary \bean \label{3} \frac {b(z)}{b(x)}\ \leq
      \frac{q+j}{q+1+j}\ \frac {a(z)}{a(x)}\ . \eean We will use also
      the recurrence relations \bea a(w)\ &=& \ (q+2+j)\,
      a(x)\ +\ (p+2+i)\, a(z)\ , \\ b(w)\ &=& \ (q+1+j)
      \,b(x)\ +\ (p+2+i)\, b(z)\ . \eea

To prove the first inequality in \Ref{2}, we need to prove that \bea \frac{a(w)}{a(z)}\ =\
(q+2+j)\,\frac {a(x)}{a(z)}\ + \ (p+2+i)\, \ \leq \ \frac{b(w)}{b(z)}\ =\ (q+1+j)\,\frac {b(x)}{b(z)}\
+ \ (p+2+i)\, \eea or \bean \label{4} \frac {a(x)}{a(z)}\ \ \leq \ \frac{q+1+j}{q+2+j} \cdot \frac
{b(x)}{b(z)}\ . \eean
Using \Ref{3} we write \bea \frac {a(x)}{a(z)}\ \leq \frac{q+j}{q+1+j} {\cdot} \frac {b(x)}{b(z)}\
\leq\ \frac{q+1+j}{q+2+j}\cdot \frac {b(x)}{b(z)}, \eea and this gives  \Ref{4}.

To prove the second inequality in \Ref{2}, we write \bea \frac{a(w)}{a(x)}\ =\ (q+2+j)\ + \
(p+2+i)\,\frac {a(z)}{a(x)} \eea and \bea \frac{q+2+j}{q+1+j}\cdot \frac{b(w)}{b(x)}\ =\ (q+2+j)\ +\
(p+2+i)\,\frac{q+2+j} {q+1+j} \cdot\frac {b(z)}{b(x)}\ . \eea Using inequality \Ref{3} we continue \bea
 &&
(q+2+j) + (p+2+i)\cdot\frac{q+2+j}{q+1+j}\cdot\frac {b(z)}{b(x)}\
 \\
&& \phantom{aaa} \leq (q+2+j) + (p+2+i)\cdot\frac{q+2+j}{q+1+j}\cdot \frac{q+j}{q+1+j}\cdot \frac
{a(z)}{a(x)}\
\\
 &&
\phantom{aaa} \leq \ (q+2+j) + (p+2+i)\cdot \frac {a(z)}{a(x)} \ = \ \frac {a(w)}{a(x)}\ . \eea Thus,
we get \bea \frac{q+2+j}{q+1+j}\cdot \frac{b(w)}{b(x)}\ \leq\ \frac {a(w)}{a(x)}\ , \eea which is the
second inequality in \Ref{2}. \qed

\begin{rem}
{ It is interesting that in order
to prove monotonicity in each coordinate
 direction, we have to consider both directions simultaneously
and we have to involve a speed in one of the two directions.}
\end{rem}

\section{Limit theorems for generalized Eulerian numbers}\
\label{Sec limit theorems} {To study asymptotics of the ratios
 discussed in the preceding section, we
first determine their limits in the coordinate directions.}
\begin{prop}
\label{prop limit}
 For fixed $i \geq 0$,
 \be\label{first}
\frac{\BB (i,j)}{A_{p,q-1}(i,j)} \searrow \frac{p+q+i+1}{p+q+1} \quad \text{ as } j \to \infty; \en
 while for fixed $j\geq 0$,
\be\label{second} \frac{\BB (i,j)}{A_{p-1,q}(i,j)} \searrow \frac{p+q+j+1}{p+q+1} \quad \text{ as } i
\to \infty. \en
\end{prop}

\begin{proof}
The dominant term in both the numerator and denominator of (\ref{first}) occurs when $t=i$ in Formula
\Ref{BB nw}, and the quotient of the two is $(p+q+i+1)/(p+q+1)$. {The fact that this
limit} is a decreasing limit follows from Theorem \ref{thm monotonicity} .

For Formula (\ref{second}), we interchange the roles of $p$ and $q$ and use Corollary \ref{cor
interchange i,j}.
\end{proof}

\begin{thm}
\label{most} Both ratios \be \frac{\BB (i,j)}{A_{p,q-1}(i,j)} \qquad \text{ and} \qquad \frac{\BB
(i,j)}{A_{p-1,q}(i,j)} \en tend to $\infty$ as {both
 $i,j$ tend to infinity;}
that is, given $M$ there are $I,J$ such that each ratio is greater than $M$ whenever $i\geq I$ and $j
\geq J$.
\end{thm}

\begin{proof}
For the first ratio, given $M$ choose $I$ so that $(p+q+I-1)/(p+q-1)
> M$. Then ${\BB (i,j)}/{A_{p,q-1}(i,j)} > M$ for all $i,j$ with $i
\geq I$ according to Proposition \ref{prop limit} and Corollary \ref{cor monot}. A similar argument
applies to the second ratio.
\end{proof}

\section{Good paths}
\label{sec good}

For any vertex  of the Euler graph, we fix an order on the set of all horizontal edges exiting that
vertex and an order on the set of all vertical edges exiting that vertex.

Let $\cs_{p,q}(i,j)$ be the set of all paths from $(p,q)$ to $(p+i,q+j)$.  We will define a subset
$\cg_{p,q}(i,j) \subset \cs_{p,q}(i,j)$ of { good} paths.  The number of elements in $\cs_{p,q}(i,j)$
is the Eulerian number $\BB(i,j)$. The number of elements in $\cg_{p,q}(i,j)$ will be denoted by
$G_{p,q}(i,j)$.

Fix $(p,q)$ and call it a {\em base point}. For any $k,l\geq 0$, the
vertex $(p+k,q+l)$ has $p+q+k+l+2$ edges leaving it, $q+l+1$
horizontally and $p+k+1$ vertically.  We label the first $q+1$
horizontal edges by symbols $s_1,\dots,s_{q+1}$, respectively, and the
first $p+1$ vertical edges by symbols $s_{q+2},\dots,s_{p+q+2}$,
respectively.

A path $x\in \cs_{p,q}(i,j)$ is a sequence of edges each of which is
labeled or not.  A path is called {\it good} if its edges have each of
the labels $s_1,\dots,s_{p+q+2}$ at least once.

The subset $\cg_{p,q}(i,j)$ of good paths is nonempty if and only if
$i\geq q+1$ and $j\geq p+1$.

\begin{thm}
\label{thm on good}
For fixed $(p,q)$, \bean \label{form g=a}
\frac{G_{p,q}(i,j)}{A_{p,q}(i,j)}\ \to\ 1 \eean as both $i,j$ tend to
infinity.
\end{thm}

\begin{proof}
It is clear that the number of elements of the set $\cs_{p,q}(i,j) \setminus \cg_{p,q}(i,j)$ of bad
paths is not greater than { $(q+1)A_{p,q-1}(i,j)+(p+1)A_{p-1,q}(i,j)$. By Theorem \ref{most}, the ratio
\bea \frac{(q+1)A_{p,q-1}(i,j)+(p+1)A_{p-1,q}(i,j)}{A_{p,q}(i,j)} \to 0 \quad\text{as}\quad i,j \to
\infty\ .
 \eea
This implies the theorem.}
\end{proof}

\section{A one-to-one correspondence between two sets of good paths}
\label{sec identif}

\begin{thm}
\label{thm good are equal} The number $G_{p,q}(i-p,j-q)$ of good
paths from $(p,q)$ to $(i,j)$ is equal to the number
$G_{p',q'}(i-p',j-q')$ of good paths from $(p',q')$ to $(i,j)$, if
$p+q=p'+q'$ and $i,j \geq p+q+2$.
\end{thm}

\begin{proof}
Denote $n=p+q=p'+q'$.

Take $(p,q)$ as a base point. For any $k,l\geq 0$, using that base
point define (as in Section \ref{sec good}) the labeled edges
$s_1,\dots,s_{n+2}$ exiting {any} vertex $(p+k,q+l)$.

Similarly, take $(p',q')$
as a base point.  For any $k,l\geq 0$, using that base point define (as in
Section \ref{sec good} the labeled edges $s_1,\dots,s_{n+2}$ exiting
{any} vertex $(p'+k,q'+l)$.

For any good path $x\in \cg_{p,q}(i-p,j-q)$ from $(p,q)$ to $(i,j)$
 (with respect to the first labels) we will construct a good path
 $y\in \cg_{p',q'}(i-p',j-q')$ from $(p',q')$ to $(i,j)$ (with respect
 to the second labels). This construction will establish a bijection
 between the corresponding sets of good
{ paths.}

{
Let $x=(E_1,E_2,\dots,E_r)$, $r=i+j-p-q$,}
be a path from $x_0=(p,q)$ to $x_r=(i,j)$, where $E_m$ are
edges and for every
 $m$, the edge $E_m$ connects a vertex $x_{m-1}$ of level $p+q+m-1$ to a vertex
$x_m$ of level { $p+q+m$. We will} encode the path $x$ by a new sequence of symbols
$D(x)=(D_1,\dots,D_r)$, called the {\em encoding sequence} of $x$, with each
{label} $D_i$ coming from
an alphabet \be \cd=\{s_1,s_2,\dots ; h_1,h_2,\dots ; v_1,v_2, \dots\}, \en
 as follows.

For any $m$ and any $a$, if one of the edges $E_1,\dots,E_m$ has label
$s_a$, then we unlabel the edge exiting $x_{m}$ with label $s_a$. This
procedure decreases the number of labeled edges exiting $x_{m}$.
{The
edges exiting $x_{m}$ which remain labeled will be called the {\em marked} edges.
The
 edges exiting $x_{m}$ which lost a label or were initially unlabeled will be
called the {\em unmarked} edges.}

{
 For any $m$, we set $D_m=D(E_m)$ to be $s_a$ if $E_m$ is a marked edge with label $s_a$.
 { Now we label the unmarked edges.}
 We
set $D_m$ to be $h_a$ if $E_m$ is the $a$'th horizontal unmarked edge among the set of horizontal
unmarked edges. We set $D_m$ to be $v_a$ if $E_m$ is the $a$'th vertical unmarked edge among the set of
vertical unmarked edges. (Recall that the set of horizontal edges exiting each vertex is ordered and
the set of vertical edges exiting each vertex is ordered, so that the notion of the $a$'th edge among
the unmarked horizontal (or vertical) edges is well defined.) }

\begin{lem}
\label{lem on D}
Let $x$ be a path from $(p,q)$ to $(i,j)$. Let $D(x)=(D_1,\dots,D_r)$ be its encoding
sequence. Then for any $m\geq 1$, we have the following two statements:
\begin{enumerate}
\item
The number of horizontal unmarked edges exiting $x_m$ equals the sum of the number
{of
marked}
edges among $E_1,\dots,E_m$ and the number of vertical unmarked edges among $E_1,\dots,E_m$.
\item
The number of the vertical unmarked edges exiting $x_m$ equals the sum of the number
{of
marked} edges among $E_1,\dots,E_m$ and the number of horizontal
unmarked edges among $E_1,\dots,E_m$.
\end{enumerate}
\end{lem}

\begin{proof}
 For any $m$, denote by $H_m$ and $V_m$ the numbers of unmarked horizontal and
unmarked vertical edges leaving $x_m$, respectively. {
 If $E_m$ is a marked edge, then
$H_m=H_{m-1}+1$ and $V_m=V_{m-1}+1$.  If $E_m$ is a horizontal unmarked edge, then $H_m=H_{m-1}$ and
$V_m=V_{m-1}+1$.  If $E_m$ is a vertical unmarked edge, then $H_m=H_{m-1}+1$ and $V_m=V_{m-1}$. }
\end{proof}

Similarly to the above construction, for any path $y$
{ from $(p',q')$ to $(i,j)$}
we can define its encoding sequence $D(y)=(D_1,\dots,D_r)$,
using the labels with respect to the base point $(p',q')$.  Again
every $D_m$ is $s_a,h_a$ or $v_a$ for a suitable $a$.

\begin{lem}\label{lem well defined}
There is a bijection
\bea
B\ :\ \cg_{p,q}(i-p,j-q)\ \to\ \cg_{p',q'}(i-p',j-q')\ , \quad
 x\ \mapsto\ y\ ,
\eea
that is well defined by choosing $y$ to satisfy the condition $D(y)=D(x)$.
\end{lem}
\begin{proof}

Let $x=(E_1,\dots ,E_r)$ be a path from
$x_0=(p,q)$ to $x_r=(i,j)$ with encoding sequence $D(x)=(D_1,\dots,D_r)$.
We need to show that there exists a unique path
$y=(E_1',\dots ,E_r')$ from $y_0=(p',q')$ to $y_r=(i,j)$ with encoding sequence
$D(y)$ such that $D(y)=D(x)$.
We  prove the existence of edges $E'_m$ by induction on $m$.

 All edges exiting $x_0$ and $y_0$ are marked. In both cases
the marks are $s_1,\dots,s_{n+2}$
where $n=p+q=p'+q'$. If $E_1$ has a mark $s_a$, then $E'_1$ is chosen to be the edge
exiting $y_0$ with mark $s_a$.

Assume that for some $m>1$ a path
$(E'_1,\dots,E'_{m-1})$ from $y_0$ to $y_m$ is constructed so that
$D(E'_1,\dots,E'_{m-1})=D(E_1,\dots,E_{m-1})$. By Lemma \ref{lem on
  D}, the vertices $y_m$ and $x_m$ have the same number of exiting
unmarked  horizontal edges, the same number of exiting unmarked vertical edges, and the
same number of exiting marked edges. Moreover, the exiting edges from $x_m$ and $y_m$
have exactly the same set of labels.
 Hence, for any $D(E_m)$ there exists a unique edge $E'_m$ exiting
$y_m$ with $D(E'_m)=D(E_m)$.

Thus, there exists a unique path $y=(E_1',\dots ,E_r')$ from $y_0=(p',q')$
such that $D(y)=D(x)$. It is easy to see that $y$ ends at $(i,j)$.
\end{proof}

{
 Lemma \ref{lem well defined} implies Theorem \ref{thm good are equal}. }
\end{proof}

\begin{thm}
\label{last thm} \bea \label{formula last}
\frac{\BB (i,j)}{A_{0,0}(p+i,q+j)}
\quad {\rm tends\ to}
\quad
{\frac1{(p+q+1)!}}
\eea
 as both $i,j$ tend to infinity.
\end{thm}

{
This theorem is a direct corollary of Theorems \ref{thm on good}
and \ref{thm good are equal}. Theorem \ref{last thm} implies
Theorem \ref{mainth}.}

\medskip
\begin{rem}
{
In this remark, we explain briefly the statements in the introduction about generic
points and directional unique ergodicity. If $x=\{(i_n,j_n), n\geq 0\}$ is generic for a
fully supported measure $\mu$, then $i_n,j_n$
must be unbounded, since otherwise using $x$ in Formula
(\ref{dimlim}) will assign measure 0 to many cylinders. Conversely,
let $x$ be any path  with $i_n,j_n$ unbounded. Let $C$ be the cylinder
determined by an initial path of length $n_0$. Then
\be
\lim_{n  \to \infty}
\frac{\dim((i_{n_0},j_{n_0}),(i_n,j_n))}{\dim(R,(i_n,j_n))}\ =\
\mu(C)=\frac{1}{(n_0+1)!} .
\en
Thus, each path with unbounded $i_n,j_n$
is generic for the symmetric measure.
}

Moreover, it is not necessary to speak of
{paths.
 If $C$ is a cylinder set with terminal vertex $P$ at level $n_0$, then
given $\epsilon > 0$ there is $M$ such that
\be
|\frac{\dim(P,(i,j))}{\dim(R,(i,j))}-\frac{1}{(n_0+1)!}| <
\epsilon
\en
for all $i,j \geq M$.}
\end{rem}

\medskip
\begin{rem} The approach presented here to prove ergodicity and unique fully supported
ergodicity of the symmetric measure on the Euler adic system was
developed so as to apply to the case of the higher-dimensional Euler
adics. It seems that the
{correspondence of good paths and monotonicity
arguments extend readily. An important step is to develop}
necessary formulas extending those for the $A_{p,q}(i,j)$.
\end{rem}

\end{document}